\documentclass{jT}

\RequirePackage{fix-cm}
\usepackage{amssymb,amsmath,latexsym,amsthm}
\usepackage[english]{babel}
\usepackage[fleqn,tbtags]{mathtools}
\usepackage{algpseudocode}  
\usepackage{algorithm} 

\newtheorem{theorem}{Theorem}[section]
\newtheorem{lemma}{Lemma}[section]
\newtheorem{proposition}{Proposition}[section]

\newtheorem{prop}[theorem]{Proposition}

\theoremstyle{definition}
\newtheorem{definition}{Definition}

\newtheorem{problem}{Problem}

\numberwithin{equation}{section}

\begin{document}

\title[Rarified $b$-multiplicative sequences]{Counting solutions without zeros or repetitions of a linear congruence and rarefaction in $b$-multiplicative sequences.}

\author{\sc Alexandre AKSENOV}
\address{Alexandre AKSENOV\\
Institut Fourier, UMR 5582\\
100, rue des Maths, BP $74$\\
38402 St Martin d'H\`eres Cedex, France}
\email{alexander1aksenov@gmail.com}

\newcommand{\F}{\mathbb F_p}
\newcommand{\Fx}{\mathbb F_p^\times}
\newcommand{\R}{\mathbb R}
\newcommand{\N}{\mathbb N}
\newcommand{\Z}{\mathbb Z}
\newcommand{\C}{\mathbb C}
\newcommand{\Q}{\mathbb Q}
\newcommand{\pl}{\oplus}
\newcommand{\Minus}{\frac{\phantom{xxx}}{\phantom{xxx}}\,}
\newcommand{\Tr}{\mathop{\rm Tr}}
\newcommand{\Cr}{\circ}
\newcommand{\Kv}{\blacklozenge}
\newcommand{\LINEFOR}[2]{\State\textbf{for }{#1}\textbf{ do }{#2}\textbf{ end for}}

\subjclass[2010]{05 A 10, 05 A 18, 11 B 39, 11 R 18}

\maketitle

\begin{resume}

Pour une suite fortement $b$-multiplicative donn\'ee et un nombre premier  $p$ fix\'e, l'\'etude de la $p$-rar\'efaction consiste \`a caract\'eriser le comportement asymptotique des sommes des premiers termes d'indices multiples de $p$. 
Les valeurs enti\`eres du polyn\^ome 
\flqq norme\frqq\  trivari\'e $\mathcal N_{p,i_1,i_2}(Y_0,Y_1,Y_2)\!:=\!\prod_{j=1}^{p-1}\left(Y_0{+}\zeta_p^{i_1j}Y_1{+}\zeta_p^{i_2j}Y_2\right)\!,$ o\`u 
$i_1,i_2{\in}\{1,2,\dots,p{-}1\}$,  $\zeta_p$ est une racine $p$-i\`eme primitive de l'unit\'e, d\'eterminent ce comportement asymptotique. 
On montre qu'une m\'ethode combinatoire s'applique \`a $\mathcal N_{p,i_1,i_2}(Y_0,Y_1,Y_2)$ qui permet d'\'etablir de nouvelles relations fonctionnelles entre les coefficients de ce polyn\^ome \flqq norme\frqq, diverses propri\'et\'es des coefficients de  $\mathcal N_{p,i_1,i_2}(Y_0,Y_1,Y_2)$, notamment pour $i_1{=}1,i_2{=}2,3;$ cette m\'ethode fournit des relations entre les coefficients binomiaux, de nouvelles preuves des deux identit\'es $\prod_{j=1}^{p-1}\left(1{+}\zeta_p^j{-}\zeta_p^{2j}\right){=}L_p$ (le $p$-i\`eme nombre de Lucas) et $\prod_{j=1}^{p-1}\left(1{-}\zeta_p^j\right){=}p$, le signe et le r\'esidu modulo $p$ des polyn\^omes sym\'etriques des $1{+}\zeta_p{-}\zeta_p^2$. Une m\'ethode algorithmique de recherche des coefficients de $\mathcal N_{p,i_1,i_2}$ est d\'evelopp\'ee.

\end{resume}

\begin{abstr}

Consider a strongly $b$-multiplicative sequence and a prime  $p$. Studying its $p$-rarefaction consists in characterizing the asymptotic behaviour of the sums of the first terms indexed by the multiples of $p$. The integer values of 
the ``norm" $3$-variate polynomial $\mathcal N_{p,i_1,i_2}(Y_0,Y_1,Y_2)\!:=\!\prod_{j=1}^{p-1}\left(Y_0{+}\zeta_p^{i_1j}Y_1{+}\zeta_p^{i_2j}Y_2\right),$  where $\zeta_p$ is a primitive $p$-th root of unity, and $i_1,i_2{\in}\{1,2,\dots,p{-}1\},$  determine this asymptotic behaviour. It will be  shown that a combinatorial method can be applied to $\mathcal N_{p,i_1,i_2}(Y_0,Y_1,Y_2).$ The method enables deducing functional relations between the coefficients as well as various properties of the coefficients of $\mathcal N_{p,i_1,i_2}(Y_0,Y_1,Y_2)$, in particular for $i_1{=}1,i_2{=}2,3.$
This method provides relations between binomial coefficients. It gives new proofs of the two identities $\prod_{j=1}^{p-1}\left(1{-}\zeta_p^j\right){=}p$ and $\prod_{j=1}^{p-1}\left(1{+}\zeta_p^j{-}\zeta_p^{2j}\right){=}L_p$ (the $p$-th  Lucas number). The sign and the residue modulo $p$ of the symmetric polynomials of $1{+}\zeta_p{-}\zeta_p^2$ can also be obtained. An algorithm for computation of coefficients of $\mathcal N_{p,i_1,i_2}(Y_0,Y_1,Y_2)$ is developed.

\end{abstr}

\bigskip

\section{Introduction}

This article deals with a combinatorial method adapted to the coefficients of homogeneous $3$-variate ``norm" polynomials which determine the asymptotic behaviour of rarified sums of a strongly $b$-multiplicative sequence. 
The general definition of a strongly $b$-multiplicative sequence of complex numbers (see \cite{Alkauskas}) can be written as:
\begin{definition}
\label{DefStronglyMul}
Let $(t_n)_{n\geqslant0}$ be a sequence of complex numbers and $b\geqslant2$  an integer. The sequence $(t_n)_{n\geqslant0}$ is called \it strongly b-multiplicative \rm if it satisfies, for each $n\in\N$, the equation
\[t_n=\prod_{i=0}^l t_{c_i},\]
where $n=\sum_{i=0}^l c_ib^i$ is the $b$-ary expansion of a natural integer $n$. Additionally, we ask that $t_0=1$ or $t_n$ is identically zero.
\end{definition}
This definition ensures that $t_n$ does not depend on the choice of the $b$-ary expansion of $n$ (for $b$-ary expansions which may or may not start with zeroes).
If the values of a strongly $b$-multiplicative sequence are either $0$ or roots of unity, it is $b$-automatic. 
An example of such sequence is the $\{1,-1\}$-valued Thue-Morse sequence defined by $b=2,t_1=-1$ (referred as $A106400$ in OEIS, cf \cite{Sloane}). Further in this text we are going to refer to this sequence as the \it Thue-Morse sequence. \rm 
A survey on the strongly $b$-multiplicative sequences with values in an arbitrary compact group can be found in \cite{Morgenbesser}.

Rarified sums (or $p$-rarified sums, the term is due to \cite{DrmotaSkalba2}) of a sequence $(t_n)_{n\geqslant0}$ are the sums of initial terms of the subsequence $(t_{pn})_{n\geqslant0}$ (the \it rarefaction step  \rm $p$ is supposed to be a prime number in this paper). The problem of estimating the speed of growth of these sums has been studied  in \cite{Gelfond},\cite{Dekking},\cite{GKS},\cite{Grabner1993},\cite{Hofer}.
The following result has been proved in a special case.
\begin{proposition}[see \cite{GKS}, Theorem $5.1$] 
\label{PropGKS}
Let $(t_n)_{n\geqslant0}$ be the Thue-Morse sequence. Suppose that $b=2$ is a generator of the multiplicative group $\Fx$. Then,
\begin{equation}
\label{eqGKS}
\sum_{n<N,p \,\mid\, n}t_n=O\left(N^\frac{\log p}{(p-1)\log 2}\right) 
\end{equation}
and this exponent cannot be decreased.
\end{proposition}

We are going to study this problem in a more general case. In Section \ref{secSommesPrar} we generalize Proposition \ref{PropGKS} to Proposition \ref{PropAsymptotique} valid for a large subclass of strongly $b$-multiplicative sequences (with different values of $b$). Proposition \ref{PropAsymptotique} describes the speed of growth of $p$-rarified sums of a strongly $b$-multiplicative sequence in a form  
\begin{equation}
\label{eqThese0}
\sum_{n<N,p \,\mid\, n}t_n=O\left(N^{\frac1{(p-1)\log b}{\log \left(\mathbf{N}_{\Q(\zeta_p)/\Q}(\sum_{c=0}^{b-1} t_c\zeta_p^c)\right)}}\right) 
\end{equation}
similar to \eqref{eqGKS}. 
The more general formula \eqref{eqThese0} contains the quantity 
\begin{equation}
\label{normXi}
\xi\left((t_n)_{n\geqslant0},p\right):=\mathbf{N}_{\Q(\zeta_p)/\Q}\left(\sum_{j=0}^{b-1} t_j\zeta_p^j\right).
\end{equation}
The properties of this norm expression provide information about the speed of growth of rarified sums. 
In Sections \ref{secCombin}, \ref{secPeq}, \ref{secArray} we develop a method to study these norms.

Denote by $d((t_n)_{n\geqslant0})$ the number of nonzero terms among $t_1,\dots,t_{b-1}$. For technical reasons, the method described in this article concerns only the strongly $b$-multiplicative sequences such that $d((t_n)_{n\geqslant0}){\leqslant}2$; in general, if $d\geqslant3$, it leads to too difficult computations. On the other hand, the case where $d((t_n)_{n\geqslant0}){=}1$ (which concerns, for example, the Thue-Morse sequence) is relatively easy. In the short description of the method, which follows, we assume that $d((t_n)_{n\geqslant0}){=}2$.

Our method consists in dealing with a strongly $b$-multiplicative sequence of monomials instead of the initial strongly $b$-multiplicative sequence of complex numbers. Let $i_1,i_2{\in}\{1,\dots,b{-}1\}$ be the two indices such that $t_{i_1}{\ne}0$ and $t_{i_2}{\ne}0$. Then the strongly $b$-multiplicative sequence of monomials associated with the sequence $(t_n)_{n\geqslant0}$ and the choice of the order of $i_1,i_2$ is defined by
\begin{align*}
\label{seqMonom}
T_0&=1,&\\
T_{i_1}&=Y_1,&\\
T_{i_2}&=Y_2,&\\
T_c&=0\text{ if }c\in\{1,\dots,b{-}1\}\setminus\{i_1,i_2\}\\
T_n&=\prod_{i=0}^l T_{c_i}\text{ otherwise,} 
\end{align*}
where $n=\sum_{i=0}^l c_ib^i$ is the $b$-ary expansion of a natural integer $n$.

Clearly, $T_n{=}0$ if and only if $t_n{=}0$. For example, if $b{=}3$ and $i_1{=}1,i_2{=}2$ then the sequence $(T_n)_{n\geqslant0}$ starts with
\[1,Y_1,Y_2,Y_1,Y_1^2,Y_1Y_2,Y_2,Y_1Y_2,Y_2^2,\dots\]
If $b{=}5$ and $i_1{=}1,i_2{=}2$, it starts with
\[1,Y_1,Y_2,0,0,Y_1,Y_1^2,Y_1Y_2,0,0,Y_2,\dots\]

One can define the $p$-rarefied sums of the sequence $(T_n)_{n\geqslant0}$ and the formal object
\begin{equation}
\label{NormFormelle}
\prod_{j=1}^{p-1}\left(\sum_{c=0}^{b-1} \zeta_p^{jc}T_c\right)
\end{equation}
which plays the role of $\xi\left((T_n)_{n\geqslant0},p\right)$. One can write \eqref{NormFormelle} explicitly as
\begin{equation}
\label{normPolynom0}
\bar{\mathcal N}_{p,i_1,i_2}(Y_1,Y_2)=\prod_{j=1}^{p-1}\left(1+\zeta_p^{i_1j}Y_1+\zeta_p^{i_2j}Y_2\right),
\end{equation}
and homogenize this polynomial, which defines
\begin{equation}
\label{normPolynom}
\mathcal N_{p,i_1,i_2}(Y_0,Y_1,Y_2)=\prod_{j=1}^{p-1}\left(Y_0+\zeta_p^{i_1j}Y_1+\zeta_p^{i_2j}Y_2\right).
\end{equation}
The norm \eqref{normXi} is then recovered as the value $\mathcal N_{p,i_1,i_2}(1,t_{i_1},t_{i_2}).$ By definition, $\mathcal N_{p,i_1,i_2}(Y_0,Y_1,Y_2)$ is the norm of $(Y_0+\zeta_p^{i_1}Y_1+\zeta_p^{i_2}Y_2)$ as a polynomial in the $4$ variables $Y_0,Y_1,Y_2,\zeta_p$ relative to the extension of fields $\Q(\zeta_p)/\Q$ in the sense of the extended definition of norm introduced in \cite{Trager}.

The form \eqref{normPolynom} of ``norm" polynomial reveals to be common for the strongly $b$-multiplicative sequences which satisfy $d((t_n)_{n\geqslant0}){=}2$ (as defined above). In order to retrieve a particular sequence from this form, one should set the formal variables $Y_0,Y_1,Y_2$ to special values, fix the two residue classes $i_1,i_2$ and take a base $b$ (bigger than the smallest positive representatives of $i_1,i_2,$ and such that the residue class of $b$ modulo $p$ is in $\Fx$, and it generates this multiplicative group). Since the form \eqref{normPolynom} inherits the properties of its coefficients, any functional relation between these coefficients can be considered as a key result.

In this context, Sections \ref{secCombin} and \ref{secPeq} enunciate a combinatorial interpretation of the coefficients of \(\mathcal N_{p,i_1,i_2}\) in terms of the following counting problem.
\begin{problem}\label{pb-combin-form}
Let $p$ be a prime number, let $\mathbf{f}$ be a vector of length $p{-}1$, all elements of which are residue classes modulo $p$ among $0,i_1,i_2$  
(i.e., $\mathbf{f}{\in}\{0,i_1,i_2\}^{p{-}1}{\subset}\F^{p-1}$). Let $i$ be an element of $\F.$ Find the number of vectors $\mathbf{x}{\in}\F^{p-1}$ which are permutations of $(1,2,\dots,p-1)$ and such that
\[\mathbf{f}\cdot\mathbf{x}=i.\]
\end{problem}
The equivalence of Problem \ref{pb-combin-form} and the problem of determining the coefficients of \(\mathcal N_{p,i_1,i_2}\) is made explicit in Proposition \ref{prDeltaSym}.   

Our main result is the following.
\begin{theorem}[equivalent to Theorem \ref{Pascal}]\label{Pascal0}
Let $p$ be an odd prime, and $i_1,i_2\in\Fx$  such that $i_1\ne i_2$. Denote by $\triangle^{i_1,i_2}(n_1,n_2,p)$ the coefficient of the term $Y_0^{p-1-n_1-n_2}Y_1^{n_1}Y_2^{n_2}$ in $\mathcal N_{p,i_1,i_2}(Y_0,Y_1,Y_2)$. Fix exponents $n_1,n_2\in\{1,\dots,p-2\}$ such that $n_1+n_2<p.$ Then,
\begin{multline}
\label{PascalAp0}
\triangle^{i_1,i_2}(n_1,n_2,p)\equiv-\triangle^{i_1,i_2}(n_1-1,n_2,p)-\triangle^{i_1,i_2}(n_1,n_2-1,p)\text{ \rm mod }p
\end{multline}
and if  $p\nmid n_1i_1+n_2i_2$, the equality
\begin{multline}
\label{PascalA0}
\triangle^{i_1,i_2}(n_1,n_2,p)=-\triangle^{i_1,i_2}(n_1-1,n_2,p)-\triangle^{i_1,i_2}(n_1,n_2-1,p)
\end{multline}
holds.
\end{theorem}
The relation \eqref{PascalA0} (similar to the recurrence equation of the Pascal's triangle) can be used to find closed formulas for some classes of coefficients (for \it all \rm of them in the case $i_1=1,i_2=2$) and to find the remaining coefficients in a fast algorithmic way. A closed formula for these coefficients is a final goal.

In Section \ref{secArray} we describe an algorithm in $O(p^2)$ additions that calculates the coefficients of \eqref{normPolynom} using this relation. 
We study the case $i_1=1,i_2=2$ and re-prove the result 
\begin{equation}
\label{normLucas0}
\prod_{j=1}^{p-1}\left(1{+}\zeta_p^j{-}\zeta_p^{2j}\right){=}L_p,
\end{equation}
the $p$-th Lucas number (i.e., the $p$-th term of the sequence referred as $A000032$ by OEIS, cf \cite{Sloane}). We formulate two corollaries of the new proof. We also state some results about the case $i_1=1,i_2=3$.

Throughout the paper, $|X|$ and $\#X$ will both refer to the size of a finite set $X$; the symbol $\#$ followed by a system of equations, congruences or inequalities will denote the number of solutions; and $\sum X$, standing for $\sum_{x\in X}x$, will refer to the sum of a finite subset $X$ of a commutative group with additive notation.

The results of this article (except Theorem \ref{thAnQualitePreuve} and Subsection \ref{subsecPascal13}) are part of the Ph.D. thesis \cite{these}.

\section{Partial sums of a stronly $q$-multiplicative sequence.}
\label{secSommesPrar}
We are going to prove an asymptotic result about partial sums of a strongly $q$-multiplicative sequence used in the proof of Proposition \ref{PropAsymptotique}.
\begin{lemma}
\label{PropSPartielles}
Let $q{\geqslant}2$ be an integer and consider a strongly $q$-multiplicative sequence $(\tau_n)_{n\geqslant0}$ of complex numbers of absolute value smaller than or equal to $1$. Denote the partial sums of $(\tau_n)_{n\geqslant0}$ by
$$\psi(N):=\sum_{n<N}\tau_n\ (N\in\N).$$
Then we have the following.

If $|\psi(q)|\leqslant1$, then 
\begin{equation}
\label{SommePartPetite}
\psi(N)=O(\log N).
\end{equation}

If $|\psi(q)|>1$, then 
\begin{equation}
\label{SommePartGrande}
\psi(N)=O\left(N^{\frac1{\log q}{\left(\log \left|\sum_{c=0}^{q-1}\tau_c\right|\right)}}\right).
\end{equation}
\end{lemma} 
\begin{proof}
Denote, for any $N,Q\in\N$ ($Q>0$), 
$$\eta(N,Q):=Q\left\lfloor\frac{N}{Q}\right\rfloor.$$
If $Q=q^m$, then the $q$-ary expansion of $\eta(N,Q)$ can be obtained from the $q$-ary expansion of $N$ by replacing the last $m$ digits by zeroes.

Suppose that $N$ is a natural integer with $q$-ary expansion $N=\sum_{i=0}^l c_iq^i$. Then,
\begin{align}
\label{Spsi}
&\psi(N){=} \sum\limits_{n=0}^{\eta(N,q^l)-1}\tau_n{+} \sum\limits_{n=\eta(N,q^l)}^{\eta(N,q^{l-1})-1}\tau_n{+}\dots{+}
\sum\limits_{n=\eta(N,q)}^{N-1}\tau_n\notag\\
&{=}\left(\sum\limits_{c=0}^{c_l-1}\tau_c\right)\left(\sum\limits_{c=0}^{q-1}\tau_c\right)^l{+}\tau_{c_l}\left(\sum\limits_{c=0}^{c_{l-1}-1}\tau_c\right)\left(\sum\limits_{c=0}^{q-1}\tau_c\right)^{\makebox[2pt]{\scriptsize $\,l{-}1$}}
{+}\dots{+}\prod_{k=1}^l\tau_{c_k}{\cdot}\left(\sum\limits_{c=0}^{c_0-1}\tau_c\right)\notag\\
&=\sum\limits_{i=0}^l\left(\prod\limits_{k=i+1}^l\tau_{c_k}\right){\cdot} \psi(c_i) {\cdot} \psi(q)^i.
\end{align}
If $|\psi(q)|\leqslant1$ then each term of the sum \eqref{Spsi} is bounded (by the maximum of $|\psi(c)|,c=1,\dots,q-1$), therefore $\psi(N)=O(l)=O(\log N).$

Suppose that $|\psi(q)|>1$. Then we are going to extend the definition of the function $\psi(x)$ to all real $x\geqslant0$ using the right-hand side of the formula \eqref{Spsi}. This requires to check that the result does not depend on the choice of the $q$-ary expansion of the argument.

Take $x=q^{-m}X$ where $m\in\Z,X\in\N,$ and the $q$-ary expansion of $X$ is $X=\sum_{i=0}^{m+l}c_{i-m}q^i.$ Then the two $q$-ary expansions of $x$ are $$x=\sum_{i=-m}^{l}c_{i}q^i=\sum_{i=-m+1}^{l}c_{i}q^i+(c_{-m}-1)q^{-m}+\sum_{i=-\infty}^{-m-1}(q-1)q^i.$$
We have to prove the identity
\begin{multline}
\label{IdExpressionsPsiX}
\sum\limits_{i=-m}^{l} \left(\prod_{k=i+1}^l\tau_{c_k}\right)\cdot \psi(c_i) d(q)^i=\\
\sum\limits_{i=-m+1}^l \left(\prod_{k=i+1}^l\tau_{c_k}\right)\cdot \psi(c_i) d(q)^i 
+ \left(\prod\limits_{k=-m+1}^l\tau_{c_k}\right)\cdot \psi(c_{-m}-1) \psi(q)^{-m}\\
+ \psi(q-1)\left(\sum\limits_{i=-\infty}^{-m-1} \left(\prod\limits_{k=i+1}^l\tau_{c_k}\right)\cdot \psi(q)^i \right).
\end{multline}
where we denote $c_i=q-1$ for $i<-m$.

Indeed, some sub-expressions of the right-hand side of \eqref{IdExpressionsPsiX} can be simplified. The last summand can be factored with one factor being
\begin{multline*}
\psi(q-1)\left(\sum\limits_{i=-\infty}^{-m-1} \left(\prod\limits_{k=i+1}^{-m-1}\tau_{c_k}\right)\cdot \psi(q)^i \right)=\psi(q-1)\sum\limits_{i=-\infty}^{-m-1}\tau_{q-1}^{-m-i-1}\psi(q)^i\\
=\psi(q-1)\tau_{q-1}^{-m-1}\left(\frac{\tau_{q-1}}{\psi(q)}\right)^m\frac1{\frac{\psi(q)}{\tau_{q-1}}-1}=\psi(q-1)\frac1{\psi(q)^m(\psi(q)-\tau-{q-1})}\\=\psi(q)^{-m}.
\end{multline*}
Next, the sum of the two last summands in \eqref{IdExpressionsPsiX}   is
\begin{multline*}
\left(\prod\limits_{k=-m+1}^l\tau_{c_k}\right)\cdot \psi(c_{-m}-1) \psi(q)^{-m}+\left(\prod_{k=-m}^l\tau_{c_k}\right)\psi(q)^{-m}\\=\left(\prod\limits_{k>-m}\tau_{c_k}\right)\psi(c_{-m})\psi(q^{-m}).
\end{multline*}
These transformations reduce the right-hand side of \eqref{IdExpressionsPsiX} to the form of the left-hand side, proving the identity. Therefore, $\psi(x)$ is a well-defined function of a real argument.

This function is continuous. Indeed, consider a sequence $(x_n)_{n}$ of positive real numbers which converges to $x>0$. Suppose that either $x_n>x$ for all $n$ or $x_n<x$ for all $n$. Let $x=\sum_{i=-\infty}^{l}c_{i}q^i$ be the $q$-ary expansion of $x$ which has a property chosen depending on the choice above: if $x_n>x,$ the expansion of $x$ does not end by $q-1'$s, if $x_n<x,$ it does not end by zeroes. In both cases, for each $m>0$ there is a rang $\tilde n$ such that $n>\tilde n$ implies that any $q$-ary expansion of $x_n$ (denote it by $x_n=\sum_{i=-\infty}^{l}{\bar c}_{i}q^i$) has all digits before radix point and $m$ digits after radix point identical to those of $x$. This property implies: 
\begin{multline}
|\psi(x)-\psi(x_n)|=\left|\sum\limits_{i=-\infty}^{m-1}\prod_{k>i}\tau_{c_k}\cdot d(c_i)d(b)^i -
\sum\limits_{i=-\infty}^{m-1}\prod_{k>i}\tau(\bar{c}_k)\cdot d(\bar c_i)d(b)^i\right|\\[5mm]
\leqslant 2\max_{c\in\{0,\dots,b-1\}} \sum_{i<-m}|d(b)|^i \xrightarrow[m\to\infty]{}0,
\end{multline}
which proves that the sequence $(\psi(x_n))_n$ converges to $\psi(x)$.

By the definition \eqref{Spsi}, for any $x\geqslant0,$ $\psi(qx)=\psi(q)\psi(x).$ Therefore,
\begin{align*}
\psi(x)&=\left(\frac{x}{q^l}\right){\psi(q)}^l\text{ where }l=\lfloor\log_qx\rfloor,\text{ therefore}\\
|\psi(x)|&\leqslant \left(\max_{x\in[1,q]} |\psi(x)|\right)|\psi(q)|^l=O(|\psi(q)|^l)=O(x^\frac{\log |\psi(q)|}{\log q}).
\end{align*}
Lemma is proved.

\end{proof}
Remark that the second part of this Theorem \eqref{SommePartGrande} has been proved in the article \cite{Grabner1993} (formula ($2.9$)).

The previous Lemma leads to the following asymptotic result about the $p$-rarified sums.
\begin{prop}
\label{PropAsymptotique}
Consider a strongly $b$-multiplicative sequence $(t_n)_{n\geqslant0}$ with values in $\{-1,0,1\}$ and a prime number $p$ such that $b<p$ is a generator of the multiplicative group $\Fx$. Suppose that the following inequality holds:
\begin{equation}
\label{CondDominantTermZeta}
\left|\mathbf{N}_{\Q(\zeta_p)/\Q}\left(\sum_{c=0}^{b-1} t_c\zeta_p^c\right)\right|>\max\left((\sum_{c=0}^{b-1}t_c)^{p-1},1\right)
\end{equation}
where $\zeta_p$ denotes a primitive $p$-th root of unity and $\mathbf{N}_{L/K}$ denotes the norm. Then we have the following estimation:
\begin{equation}
\label{eqThese2}
\sum_{n<N,p \,\mid\, n}t_n=O\left(N^{\frac1{(p-1)\log b}{\log \left(\mathbf{N}_{\Q(\zeta_p)/\Q}(\sum_{c=0}^{b-1} t_c\zeta_p^c)\right)}}\right) 
\end{equation}
\end{prop}
\begin{proof}
The norm in \eqref{eqThese2} is real and nonnegative because it is a product of $\frac{p-1}2$ complex-conjugate pairs. Furthermore, it is bigger than $1$ by the hypothesis \eqref{CondDominantTermZeta}. This proves that the right-hand side of \eqref{eqThese2} has a meaning.

The $p$-rarefied sum in the left-hand side can be expanded as
\begin{multline}
\label{Raref2zeta}
\sum_{n<N}1_{p|n} t_n
= \sum_{n<N}\frac1p\left(1+\zeta_p^n+\zeta_p^{2n}+\ldots+\zeta_p^{(p-1)n}\right)t_n
\\=\frac1p\left(\sum_{n<N}t_n+\sum_{n<N}\sum_{j\in\F^\times}\zeta_p^{jn}t_n\right). 
\end{multline}
Remark that the sequences $(t_n)_{n\geqslant0}$ and $(\zeta_p^{jn}t_n)_{n\geqslant0}$ ($j\in\{1,\ldots,p{-}1\}$), which appear in the previous formula, are strongly $b^{p-1}$-multiplicative.

By Lemma \ref{PropSPartielles}  
applied to the sequence $(t_n)_{n\geqslant0}$, one has one of the two estimations
\begin{align*}
\sum_{n<N}t_n&=O\left(N^{\frac1{\log b}\log \left|\sum_{c=0}^{b-1}t_c\right|}\right)\text{ or}\\
\sum_{n<N}t_n&=O(\log N).
\end{align*}
In both cases we get (using the hypothesis \eqref{CondDominantTermZeta}), 
\begin{equation}
\label{SumTnO}
\sum_{n<N}t_n=O\left(N^{\frac1{(p-1)\log b}{\log \left(\mathbf{N}_{\Q(\zeta_p)/\Q}(\sum_{j=0}^{b-1} t_j\zeta_p^j)\right)}}\right).
\end{equation}

Lemma  \ref{PropSPartielles}  
applied to a sequence of the form $(\zeta_p^{jn}t_n)_{n\geqslant0}$ states that
\begin{align}
\sum_{n<N}\zeta_p^{jn}t_n&=O\left(N^{\frac1{(p-1)\log b}\log \left|\sum_{c=0}^{b^{p-1}-1}\zeta_p^{j c}t_c\right|}\right)\text{ or} \label{SumNorm-polynomial-O}\\
\sum_{n<N}\zeta_p^{jn}t_n&=O(\log N).\label{Sum-O-petit}
\end{align}

On the other hand, one can expand the norm of $(\sum_{c=0}^{b-1} t_c\zeta_p^c)$ as
\begin{multline}
\mathbf{N}_{\Q(\zeta_p)/\Q}(\sum_{c=0}^{b-1} t_c\zeta_p^c)=\prod_{i=0}^{p-2}\left(\sum_{c=0}^{b-1}\zeta_p^{b^icj}t_c\right)=\\
\sum_{c_0,\dots,c_{p-2}\in\{0,\dots,b-1\}} \zeta_p^{j(c_0+b c_1+\dots+b^{p-2}c_{p-2})}t_{c_0}\dots t_{c_{p-2}}
\end{multline} 
where the sequences $(c_0,\dots,c_{p-2})$ are nothing else than all possible choices of the index $c$ when the product is expanded. The change of variable $n:=c_0+bc_1+\dots+b^{p-2}c_{p-2}$ leads to a new variable which goes through all integers from $0$ to $b^{p-1}-1$. Therefore,
\begin{equation}
\label{NormToSum}
\mathbf{N}_{\Q(\zeta_p)/\Q}(\sum_{c=0}^{b-1} t_c\zeta_p^c)=\sum_{n=0}^{b^{p-1}-1}\zeta_p^nt_n,
\end{equation}
which is the sum involved in \eqref{SumNorm-polynomial-O}.

Therefore,
\begin{equation}
\label{sumZetaJTO}
\sum_{n<N}\zeta_p^{jn}t_n=O\left(N^{\frac1{(p-1)\log b}{\log \left(\mathbf{N}_{\Q(\zeta_p)/\Q}(\sum_{c=0}^{b-1} t_c\zeta_p^c)\right)}}\right).
\end{equation}

Equations \eqref{Raref2zeta}, \eqref{SumTnO} and  \eqref{sumZetaJTO} lead to the conclusion of the Proposition.
\end{proof}

Proposition \ref{PropAsymptotique} generalizes the first part of Proposition \ref{PropGKS} (the estimation \eqref{eqGKS}) as the Thue-Morse sequence satisfies the conditions of validity of Propoition \ref{PropAsymptotique} and
we get the following:
\[\mathbf{N}_{\Q(\zeta_p)/\Q}\left(\sum_{j=0}^{b-1} t_j\zeta_p^j\right)=\mathbf{N}_{\Q(\zeta_p)/\Q}(1-\zeta_p)=p.\]

Another situation where the norms
\begin{equation}
\label{normXi2}
\mathbf{N}_{\Q(\zeta_p)/\Q}\left(\sum_{j=0}^{b-1} t_j\zeta_p^j\right)
\end{equation}
can be calculated in a straightforward way is the situation where $b=3,t_0{=}t_1{=}1,t_2{=}{-}1$. Using the resultant of the two polynomials $S(X)=X^{p-1}+\dots+1$ and $R(X){=}X^2{-}X{-}1$, one obtains 
\begin{equation}
\label{NormLucas}
\mathbf{N}_{\Q(\zeta_p)/\Q}(1+\zeta_p-\zeta_p^2)=L_p
\end{equation}
the $p$-th term of the Lucas sequence (referred as $A000032$ by OEIS, cf \cite{Sloane}) defined recursively by $L_0\!=\!2,{L_1\!=\!1},L_{n+2}=L_n{+}L_{n+1}$. This result is proved in a different way in Section \ref{SsecTriangle12}.

\section{Combinatorics of partitions of a set.}
\label{secCombin}
 In this section we are going to give an alternative proof of the formula 
\begin{equation}
\label{normUniformisante}
\prod_{j=1}^{j=p-1}\left(X-\zeta^j\right)=1+X+\dots+X^{p-1},
\end{equation}
and the methods of this proof will be re-used in the proof of the functional equation in Section \ref{secPeq}. The new proof uses the properties of  the partially ordered sets $\Pi_n$ of  partitions of a set of size $n$ (a good reference about the properties of those is the Chapter $3.10.4$ of \cite{EnumCombin}).  We are going to prove the following statement, which is equivalent to \eqref{normUniformisante}.

\begin{lemma}
\label{delta1d}
Let $p$ be a prime number and $0\leqslant n<p$ an integer. Define $A_0(n,p)$ as the number of subsets of $\Fx$ of $n$ elements that sum up to $0$ modulo $p$ and $A_1(n,p)$ the number of those subsets that sum up to $1$. Then
$$A_0(n,p)-A_1(n,p)=(-1)^n. $$ 
\end{lemma}

Let us begin the proof with an obvious observation: if we define similarly the numbers $A_2(n,p)$, $A_3(n,p)$, \dots, $A_{p-1}(n,p)$ , they will all be equal to $A_1(n,p)$, since multiplying a set that sums to $1$ by a constant residue $c\in\Fx$ gives a set that sums to $c$, and this correspondence is one-to-one.

Let us deal with a simpler version of the Lemma 
that allows repetitions and counts sequences instead of subsets, which is formalized in the following.

\begin{definition}
Denote by $E^{k_1,\dots,k_n}_x(n,p)$ (where $x\in\F$ and $k_1,\dots,k_n\in\Fx$) the number of sequences $(x_1,x_2,\dots,x_n)$ of elements of $\Fx$ such that
$$\sum_{i=1}^n k_i x_i=x. $$
\end{definition}
Then we get the following.

\begin{lemma}
\label{E}
If $n$ is even,
$$E^{k_1,k_2,\dots,k_n}_0(n,p)=\frac{(p{-}1)^n{+}p{-}1}{p}\quad \text{\rm  and }\quad E^{k_1,k_2,\dots,k_n}_1(n,p)=\frac{(p{-}1)^n-1}{p};$$
\hspace{0.5cm}if $n$ is odd,
$$E^{k_1,k_2,\dots,k_n}_0(n,p)=\frac{(p{-}1)^n{-}p{+}1}{p}\quad \text{\rm  and }\quad E^{k_1,k_2,\dots,k_n}_1(n,p)=\frac{(p{-}1)^n+1}{p}.$$
In both cases,
$$E^{k_1,k_2,\dots,k_n}_0(n,p)-E^{k_1,k_2,\dots,k_n}_1(n,p)=(-1)^n. $$ 

\end{lemma}
\begin{proof}
\it By induction on n.\rm\  For $n=0$ or $n=1$ the result is trivial. For bigger $n$ we always get:
$$E_0^{k_1,k_2,\dots,k_n}(n,p)=(n-1)E_1^{k_1,k_2,\dots,k_{n-1}}(n-1,p) $$
and
$$E_1^{k_1,k_2,\dots,k_n}(n,p)=E_0^{k_1,k_2,\dots,k_{n-1}}\textit{}(n-1,p)+(p-2)E_1^{k_1,k_2,\dots,k_{n-1}}(n-1,p), $$
since the sequences of length $n$ of linear combination (with coefficients $k_i$) equal to $x$  are exactly expansions of sequences of length $n-1$ of linear combination different from $x$, and this correspondence is one-to-one. Injecting formulas for $n-1$ concludes the induction.
\end{proof}

Now we are going to prove Lemma \ref{delta1d} for small $n$. If $n=0$ or $n=1$, Lemma is clear. For $n=2$, there is one more sequence $(x,y)\in\Fx{^2}$ that sums up to $0$, but that counts the sequences of the form $(x,x)$ which should be removed. Since $p$ is prime, these sequences contribute once for every nonzero residue modulo $p$, and removing them increases the zero's ``advantage" to $2$. Now, we have to identify 
$(x,y)$ and $(y,x)$ to be the same, so we get the difference $1$ back, establishing Lemma $1$ for $n=2$.

For $n=3$, counting all the sequences $(x,y,z)\in\Fx$ gives a difference $E_0-E_1=-1 $. The sequences $(x,x,z)$ contribute one time more often to the sum equal to $0$, so removing them adds $-1$ to the total difference. The same thing applies to sequences of the form $(x,y,y)$ and $(x,y,x)$. After removing them, we get an intermediate  difference of $-4$, but the triples of the form $(x,x,x)$ have been removed $3$ times, which is equivalent to saying they count $-2$ times. Therefore, they should be ``reinjected" with  coefficient $2$. As $p$ is prime and  bigger than $3$, the redundant triples contribute once for each nonzero residue; therefore we accumulate the difference of $-4-2=-6$. We have then to identify permutations, that is to divide the score by $6$ which gives the final result $-1$.

Here is the explicit calculation for the case $n=4$:

$$\begin{array}{ll}
1 & \text{\rm  (corresponds to $E_0(4,p)-E_1(4,p)$)}\\
+6 & 	\text{\rm  (for removing \begin{tabular}{ccc}  $(x,x,y,z)$,& $(x,y,x,z)$,&$(x,y,z,x)$,\\ $(x,y,y,z)$,&$(x,y,z,y)$,&$(x,y,z,z)$ \end{tabular})} \\[2mm]
+2\times 4 & \text{\rm  (for re-injecting \begin{tabular}{cc}$(x,x,x,y)$,& $(x,x,y,x)$,\\ $(x,y,x,x)$& and $(x,y,y,y)$\end{tabular})}\\[2mm]
+1\times 3 & \text{\rm  (for re-injecting $(x,x,y,y),(x,y,x,y)$ and $(x,y,y,x)$)}\\[2mm]
+6\times 1 & \text{\rm  (for removing $(x,x,x,x)$)}\\[2mm]
=24 &
\end {array}
$$
which is $4!$, therefore Lemma \ref{delta1d} is proved for $n=4$.

For a general $n$ we can calculate the difference between the number of sequences that sum up to $0$ and the number of those that sum up to $1$ by assigning to all sequences in $\Fx{^n}$ an intermediate coefficient equal to one, then by reducing it by one for each couple of equal terms, then increasing by $2$ for each triple of equal terms, and so on, proceeding by successive adjustments of coefficients, each step corresponding to a ``poker combination" of $n$ cards. If after adding the contributions of all the steps and the initial $(-1)^n$, we get $(-1)^nn!$, Lemma \ref{delta1d} is valid for $n$ independently from $p$ provided that $p>n$ is prime.

Let us introduce a  formalization of these concepts  using the notions exposed in \cite{Rota}. Call a \it partition \rm of the set $\{1,2,\dots,n\}$ a choice of pairwise disjoint nonempty subsets $B_1,B_2,\dots, B_c$ of $\{1,2,\dots,n\}$ of non-increasing sizes $|B_i|$, and such that $B_1\cup B_2 \cup\dots\cup B_c=\{1,2,\dots,n\} $. The set $\Pi_n$ of all partitions of  $\{1,2,\dots,n\}$ is partially ordered by reverse refinement: for each two partitions $\tau$ and $\pi$, we say that $\tau\geqslant\pi$ if each block of $\pi$ is included in a block of $\tau$. We define the M\"obius function $\mu(\hat 0,x)$ on $\Pi_n$  recursively by:\\
\quad if $x=\{\{1\},\{2\},\dots,\{n\}\}=\hat 0$, then $\mu(\hat 0,x)=1$;\\
\quad if $x$ is bigger than $\hat 0$, then 
$$\mu(\hat 0,x)=-\sum_
{\begin{array}{c}
y\in\Pi_n\\
y<x
\end{array}
}\mu(\hat 0,y). $$

By the Corollary to the Proposition $3$ section $7$ of \cite{Rota0} and the first Theorem from the section $5.2.1$ of \cite{Rota}, if $x$ is a partition of type $(\lambda_1,\dots,\lambda_n)$, then
\begin{equation}
\label{exprMu}
\mu(\hat 0,x)=\prod_{i=1}^{n}(-1)^{\lambda_i-1}(\lambda_i-1)! 
\end{equation}
This formula will be useful in Section \ref{secPeq}.

We are also going to use the following definition: let $x=(x_1,x_2,\dots,x_n)$ be a sequence of $n$ nonzero residues modulo $p$ seen as a function
$$x:\{1,2,\dots,n\}\to\Fx. $$ 
Then the \it coimage \rm of $x$ is the partition of $\{1,2,\dots,n\}$, whose blocks are the nonempty preimages of elements of $\Fx$. Now we can prove the following proposition that puts together all the previous study.
\begin{lemma}
\label{deltaInv}
The difference
$$A_0(n,p)-A_1(n,p) $$
does not depend on $p$ provided that $p$ is a prime number bigger than $n$.
\end{lemma}
\begin{proof}
We are going to describe an algorithm that computes this difference (which is the one applied earlier for small values of the argument). For each partition $x\in\Pi_n$, denote by $r_0(x,p)$ the number of sequences $(x_1,x_2,\dots,x_n)$ of elements of $\Fx$  of coimage $x$ that sum up to $0$, and denote by $r_1(x,p)$ the number of those sequences of coimage $x$ that sum up to $1$ and denote $r(x,p)=r_0(x,p)-r_1(x,p)$. Then,
$$n!(A_0(n,p)-A_1(n,p))=r(\hat 0,p). $$

Denote, for each partition $y$ of $\{1,2,\dots,n\}$, 
$$s(y,p)=\sum_{x\geqslant y} r(x,p). $$

Then, by Proposition \ref{E},
\begin{equation}
\label{DefS}
s(y,p)=(-1)^{c(y)} 
\end{equation}
where $c(y)$ is the number of blocks in the partition $y$. By the M\"obius inversion formula (see \cite{Rota}),
\begin{equation}
\label{ValS}
r(\hat 0,p)=\sum_{y\in\Pi_n} \mu(\hat 0,y)s(y,p)=\sum_{y\in\Pi_n}(-1)^{c(y)}\mu(\hat 0,y) . 
\end{equation}

If we compute this sum, we get the value of $A_0(n,p)-A_1(n,p)$ in a way that does not depend on $p$.
\end{proof}

The last move consists in proving that
\begin{equation}
\label{Pascal1dCombi}
\sum_{y\in\Pi_n}(-1)^{c(y)}\mu(\hat 0,y)=(-1)^nn!
\end{equation}
in a way that uses the equivalence with Lemma \ref{delta1d}.
This proof may seem to be artificial because it is no longer used in the Section \ref{secPeq},  and a purely combinatorial and more general proof exists: see the final formula of Chapter $3.10.4$ of \cite{EnumCombin}. 

Remark that $A_0(n,p)=A_0(n,p-1-n)$ since saying that the sum of some subset of $\Fx$ is $0$ is equivalent to saying that the sum of its complement is $0$. For the same kind of reason, $A_1(n,p)=A_{-1}(n,p-1-n)=A_1(n,p-1-n)$. 

Now we can prove Lemma \ref{delta1d} by induction on $n$. It has already been proved for small values of $n$. If $n>4$, by Bertrand's postulate, there is a prime number $p'$ such that $n<p'<2n$. Replace $p$ by $p'$ (by the proposition \ref{deltaInv} this leads to an equivalent statement), then $n$ by $p'-1-n$ (using the above remark). As $p'-1-n<n$, the step of induction is done.

This proof can be analysed from the following point of view: how fast does the number of steps of induction grow as function of $n$?  
Suppose that one step of induction reduces Lemma \ref{delta1d} for $n$ to Lemma \ref{delta1d} for the number $f(n)$ and denote by $R(n)$ the number of steps of induction needed to reach one of the numbers $0$ or $1$ (the formal definitions will follow).
We can prove then the following upper bound on $R(n)$.
\begin{theorem}
\label{thAnQualitePreuve}
Let
\[\mathrm{nextprime}(n):=\min\{p>n \mid p\text{ \rm prime}\}\]
and
\[f(n):=\mathrm{nextprime}(n)-n-1\]
for each $n\in\N$. Further, denote
$$R(n):=\min\{k \mid f^k(n)\in\{0,1\}\}.$$
This definition makes sense, for  $f(n)<n$ for each $n>1$ by the Bertrand's postulate. 

The function $R(n)$ satisfies the estimation
\begin{equation}
R(n)=O(\log\log n).
\end{equation}
\end{theorem}
\begin{proof}
Denote $\theta=0.525$. By Theorem 1 of \cite{BakerHarmanPinz}, there is a constant $N_0$ such that for all $n>N_0$, the interval $[n-n^{\theta},n]$ contains a prime number. We are going to deduce from this the following result: for each $\bar\theta\in]0.525,1[$ there exists a constant $N_1$ such that $n>N_1$ implies $f(n)<n^{\bar\theta}$.

Indeed, suppose $n>N_0$ and denote $\bar p=\mathrm{nextprime}(n)-1$. Then, by the result cited above,
\begin{equation}
\label{conseqBHP}
n\geqslant \bar p-\bar p^\theta.
\end{equation}
The function
\[
\begin{array}{rrcl}
u:&[N_0,+\infty[&\to&[u(N_0),+\infty[\\&x&\mapsto&x-x^\theta
\end{array}
\]
is strictly increasing, continuous and equivalent to $x$. Therefore,  
the same is valid for its inverse $u^{-1}$.
By \eqref{conseqBHP}, $\bar p\leqslant u^{-1}(n)$, therefore
\begin{equation}
\forall n>N_1\ f(n)=\bar p-n\leqslant\bar{p}^\theta\leqslant(u^{-1}(n))^\theta<n^{\bar\theta}
\end{equation}
for each $\bar\theta\in]\theta,1[$ and for a bound $N_1\geqslant N_0$ that may depend on $\bar\theta$.

The end of the proof is analogous to that of Theorem $1.1$ of \cite{LucaThangadurai}. Denote by $l$ the integer such that $f^{l+1}(n)<N_1\leqslant f^{l}(n)$. Then:
\[n^{{\bar\theta}^l}\geqslant N_0\]
therefore
\[l\log{\bar\theta}+\log\log n\geqslant \log\log N_1\]
which implies
\[l\leqslant-\frac{\log\log n}{\log{\bar\theta}}.\]
Put $b=\max_{1\leqslant m\leqslant N_0}R(m)$, it is a constant. We get:
\[R(n)\leqslant l+1+b\leqslant-\frac{\log\log n}{\log{\bar\theta}}+1+b\]
which proves our claim.
\end{proof}

\section{Pascal's equation.}
\label{secPeq}

We are going to prove the functional equation satisfied by the coefficients of the polynomial $\mathcal N_{p,i_1,i_2}(Y_0,Y_1,Y_2)$ (introduced in \eqref{normPolynom}). To do this, we are going to describe a combinatorial interpretation of these numbers.
\begin{definition}
\label{A}
Let $p,i_1,i_2$ be fixed as in Introduction and $n_1,n_2$ be nonnegative integers such that $n_1+n_2\leqslant p-1$.  
Define
\begin{align}
C^{i_1,i_2}_i(n_1,n_2,p)&=\label{defC}\\
&\#\left\{(x_1,\dots,x_{n_1+n_2})\in{\Fx}^{n_1+n_2}\left|
\begin{array}{@{}c@{}}x_k\ne x_l\text{ if }k\ne l,\\[2pt]
\displaystyle i_1\sum_{k=1}^{n_1}x_k+i_2\smashoperator{\sum_{k=n_1+1}^{n_1+n_2}}x_k=i\end{array}\right.\right\}\notag\\
\intertext{and}
A^{i_1,i_2}_i(n_1,n_2,p)&=\#\left\{(X_1,X_2)\in\mathcal{P}(\Fx)^2\ \left|\begin{array}{@{}c@{}}|X_1|=n_1,|X_2|=n_2,\\ X_1\cap X_2=\emptyset,\\ i_1\sum X_1+i_2\sum X_2=i\end{array} \right.\right\}.\label{defA}
\end{align} 

\end{definition}

Definition \ref{A} matches with the notations from the previous section because of the identity $A^{i_1,i_2}_i(n,0,p)=A^{i_1,i_2}_i(0,n,p)=A_i(n,p)$ (independently from $i_1,i_2$). One can also see that the answer to Problem \ref{pb-combin-form} is $(p{-}1{-}n_1{-}n_2)!C^{i_1,i_2}_i(n_1,n_2,p)$ where $n_1$ (resp., $n_2$) is the number of coordinates of the vector $\mathbf{f}$ equal to $i_1$ (resp., $i_2$). 

From this definition one can see that 
\begin{align*}
C^{i_1,i_2}_1(n_1,n_2,p)=\dots=C^{i_1,i_2}_{p-1}(n_1,n_2,p),\\
\sum_{i=0}^{p-1}C^{i_1,i_2}_i(n_1,n_2,p)=(p-1)\dots(p-n_1-n_2),
\end{align*}
and for any $i$, $A^{i_1,i_2}_i(n_1,n_2,p)=\frac{C^{i_1,i_2}_i(n_1,n_2,p)}{n_1!n_2!}$. 

Only one linear equation should be added to these in order to be able to determine all the numbers defined by \eqref{defC} and \eqref{defA}. Proposition \ref{prDeltaSym} below suggests to research the value of 
\begin{multline*}
\triangle^{i_1,i_2}(n_1,n_2,p)=A^{i_1,i_2}_0(n_1,n_2,p)-A^{i_1,i_2}_1(n_1,n_2,p)\\=
\sum_{\begin{array}{c}X_1,X_2\subset\Fx\\|X_1|=n_1, |X_2|=n_2,\\X_1\cap X_2=\emptyset\end{array}}\zeta_p^{i_1\sum X_1+i_2\sum X_2}.
\end{multline*}

We can express the symmetric polynomials of the quantities $(Y_0+\zeta_p^{i_1j}Y_1+\zeta_p^{i_2j}Y_2)$ in terms of the previously defined numbers via the following. 
\begin{prop}
\label{prDeltaSym}
Let $i_1,i_2$ be two different elements of $\Fx$ and denote by $\sigma_{v,(j=1,\dots,p-1)}$  the elementary symmetric polynomial of degree $v$ in quantities that depend on  an index $j$ varying from $1$ to $p-1$. Then we have the following formal expansion:
\begin{multline}
\label{DeltaSym}
\sigma_{p-1-\delta,(j=1,\dots,p-1)}\left(Y_0+\zeta_p^{i_1j}Y_1+\zeta_p^{i_2j}Y_2\right)=\\\sum_{\begin{array}{c}0\leqslant n_0,n_1,n_2\leqslant p-1 \\n_0+n_1+n_2=p-1\\n_0\geqslant\delta \end{array}} 
\genfrac{(}{)}{0pt}{}{n_0}{\delta}\triangle^{i_1,i_2}(n_1,n_2,p)Y_0^{n_0-\delta}Y_1^{n_1}Y_2^{n_2}.
\end{multline}
In particular,
\begin{multline}
\label{normPolyn2}
\mathcal N_{p,i_1,i_2}(Y_0,Y_1,Y_2)=\\\sum_{\begin{array}{c}0\leqslant n_0,n_1,n_2\leqslant p-1 \\n_0+n_1+n_2=p-1 \end{array}} 
\triangle^{i_1,i_2}(n_1,n_2,p)Y_0^{n_0}Y_1^{n_1}Y_2^{n_2}.
\end{multline}
\end{prop}
\begin{proof}
The symmetric polynomial develops as:
\begin{multline*}
\sigma_{p-1-\delta,(j=1,\dots,p-1)}\left(Y_0+\zeta_p^{i_1j}Y_1+\zeta_p^{i_2j}Y_2\right)
\\=\sum_{\begin{array}{c}X\subset\Fx,\\|X|=p{-}1{-}\delta\end{array}}\prod_{j\in X}\left(Y_0+\zeta_p^{i_1j}Y_1+\zeta_p^{i_2j}Y_2\right)
\\=\sum_{\begin{array}{c}X\subset\Fx,\\|X|=p{-}1{-}\delta\end{array}}\sum_{\begin{array}{c}X_0,X_1,X_2,\\X_0\cup X_1\cup X_2=X,\\X_0,X_1,X_2\text{ disjoint}\end{array}}\zeta_p^{i_1\sum X_1+i_2\sum X_2}Y_0^{|X_0|}Y_1^{|X_1|}Y_2^{|X_2|}
\\=\sum_{\begin{array}{c}X'_0,X_1,X_2,\\X'_0\cup X_1\cup X_2=\Fx\\X'_0,X_1,X_2\text{ disjoint}\end{array}}\genfrac{(}{)}{0pt}{}{|X'_0|}{\delta}\zeta_p^{i_1\sum X_1+i_2\sum X_2} Y_0^{|X'_0|-\delta}Y_1^{X_1}Y_2^{X_2}.
\end{multline*}
When we group the terms of this sum by sizes $n_0=|X'_0|,n_1=|X_1|,n_2=|X_2|$ we obtain \eqref{DeltaSym}.
\end{proof}

The method of proof of Lemma \ref{deltaInv} can be generalized into the following conditional closed formula for the coefficients $\triangle^{i_1,i_2}(n_1,n_2,p)$:
\begin{prop}
\label{propDeltaSimples}
Let $p$ be an odd prime, and $i_1,i_2\in\Fx$ , $n_1,n_2\in\{1,\dots,p-2\}$ such that $i_1\ne i_2$ and $n_1+n_2<p$. Suppose that the multiset consisting of $i_1$ with multiplicity $n_1$ and of $i_2$ with multiplicity $n_2$ should have no nonempty subset of sum multiple of $p$. Then,
\begin{equation}
\label{deltaSimples}
\triangle^{i_1,i_2}(n_1,n_2,p)=(-1)^{n_1+n_2}\genfrac{(}{)}{0pt}{}{n_1+n_2}{n_1}.
\end{equation}
\end{prop}
\begin{proof}
Define 
\begin{equation}
\label{NotationCoeffs}
f_k=\left\{\begin{array}{cl}i_1&\text{ if \(k\in\{1,\dots,n_1\}\)}\\i_2&\text{ if \(k\in\{n_1+1,\dots,n_1+n_2\}\)}\end{array}\right.
\end{equation}
and $n:=n_1+n_2 $.

Next, for each partition $x{\in}\Pi_n$ and each $i{\in}\F$ denote by $r^{i_1,i_2}_i(x,n_1,n_2,p)$ the number of sequences $(x_1,x_2,\dots,x_{n{=}n_1{+}n_2})$ of elements of $\Fx$  of coimage $x$ such that
\begin{equation}
\label{equaLin2dims}
i_1\sum_{k=1}^{n_1}x_k+i_2\smashoperator{\sum_{k=n_1+1}^{n_1+n_2}}x_k=i.
\end{equation}
Let us also define $r^{i_1,i_2}(x,n_1,n_2,p)=r^{i_1,i_2}_0(x,n_1,n_2,p)-r^{i_1,i_2}_1(x,n_1,n_2,p)$ and
\begin{align}
s^{i_1,i_2}_i(x,n_1,n_2,p)=\sum_{x'\geqslant x} r^{i_1,i_2}_i(x',n_1,n_2,p),\label{Si1i2iSumRi1i2i}\\
s^{i_1,i_2}(x,n_1,n_2,p)=s^{i_1,i_2}_0(x,n_1,n_2,p)-s^{i_1,i_2}_1(x,n_1,n_2,p).
\end{align}

By definition, 
\begin{equation}
\label{CmoinsCSimple}
C^{i_1,i_2}_0(n_1,n_2,p)-C^{i_1,i_2}_1(n_1,n_2,p)=r^{i_1,i_2}(\hat 0,n_1,n_2,p).
\end{equation}

Consider a partition $x\in\Pi_n.$ The number $s^{i_1,i_2}_i(x,n_1,n_2,p)$ (defined by the formula \eqref{Si1i2iSumRi1i2i}) admits an equivalent definition as the number of sequences $(x_1,x_2,\dots,x_n)$ of elements of $\Fx$  of coimage greater than or equal to $x$ (in the sense of partitions) which satisfy \eqref{equaLin2dims}. Denote by $B_1,\dots,B_{c(x)}$ the blocks of $x$ and
$$f_{B_j}=\sum_{k\in B_j}f_k (j=1,\dots,c(x)). $$
Then, $s^{i_1,i_2}_i(x,n_1,n_2,p)$ is the number of sequences $(x_{B_1},\dots,x_{B_{c(x)}})$ of elements of $\Fx$ (where the terms can be equal or distinct) such that
\begin{equation}
\label{equaLinBlocks}
\sum_{j=1}^{c(x)} f_{B_j} x_{B_j}=i.
\end{equation}
By the hypotheses of the Proposition, all $f_{B_j}$ are nonzero in $\F$. Therefore, by Proposition \ref{E}, 
\begin{equation}
\label{Si1i2Simple}
s^{i_1,i_2}(x,n_1,n_2,p)=(-1)^{c(x)}.
\end{equation}

By the M\"obius inversion formula and the formula \eqref{Pascal1dCombi},
$$r(\hat 0,,n_1,n_2,p)=\sum_{y\in\Pi_n}(-1)^{c(y)}\mu(\hat 0,y)=(-1)^nn!.$$
Therefore,
$$\triangle^{i_1,i_2}(n_1,n_2,p)=\frac1{n_1!n_2!}(C^{i_1,i_2}_0(n_1,n_2,p)-C^{i_1,i_2}_1(n_1,n_2,p))$$
which concludes the proof.
\end{proof}

The condition of Proposition \ref{propDeltaSimples} holds, for example, if the smallest positive representatives of $i_1$ and $i_2$ verify $n_1i_1+n_2i_2<p$.

Without the condition formulated in Proposition \ref{propDeltaSimples},  \eqref{deltaSimples} becomes false: for example, $\triangle^{2,3}(1,1,5)=-3$. 
For the general case, we are going to replace the closed formula by a recursive equation in which  the parameters $i_1,i_2,p$ are fixed, and the recursion is on different values of $n_1,n_2$. The equation is similar to the equation of the Pascal's triangle, and can be formulated as follows:
\begin{theorem}[``Colored" Pascal's equation]
\label{eqPC}
Let $p$ be an odd prime, and $i_1,i_2\in\Fx$ , $n_1,n_2\in\{1,\dots,p-2\}$ such that $i_1\ne i_2$ and $n_1+n_2<p$. Then,
\begin{multline}
\label{PascalBp}
C^{i_1,i_2}_0(n_1,n_2,p)-C^{i_1,i_2}_1(n_1,n_2,p)\equiv\\ n_1C^{i_1,i_2}_1(n_1-1,n_2,p)+n_2C^{i_1,i_2}_1(n_1,n_2-1,p)-\\-n_1C^{i_1,i_2}_0(n_1-1,n_2,p)-n_2C^{i_1,i_2}_0(n_1,n_2-1,p)\text{ \rm mod }p
\end{multline}
and if  $p\nmid n_1i_1+n_2i_2$, the equality
\begin{multline}
\label{PascalC}
C^{i_1,i_2}_0(n_1,n_2,p)-C^{i_1,i_2}_1(n_1,n_2,p)=\\n_1C^{i_1,i_2}_1(n_1-1,n_2,p)+n_2C^{i_1,i_2}_1(n_1,n_2-1,p)-\\-n_1C^{i_1,i_2}_0(n_1-1,n_2,p)-n_2C^{i_1,i_2}_0(n_1,n_2-1,p)
\end{multline}
holds.
\end{theorem}
\begin{proof}

We are going to use the notations of the beginning of the previous proof until the formula \eqref{equaLinBlocks}. We are also going to call a \it hindrance\rm\  a subset $X$ of $\{1,\dots,n_1+n_2\}$ such that $\sum_{m\in X}f_m\equiv0\text{ mod }p$. Proposition \ref{propDeltaSimples} corresponds to the case when there are no hindrances. Then
$$C^{i_1,i_2}_0(n_1,n_2,p)-C^{i_1,i_2}_1(n_1,n_2,p)=(-1)^{n_1+n_2}(n_1+n_2)! $$
and this number is the opposite  of $n$ times $(-1)^{n-1}(n-1)!$ .

In general, the formula \eqref{Si1i2Simple} should be replaced by:
\begin{equation}
\label{valSh}
s^{i_1,i_2}(y,n_1,n_2,p)=(1-p)^{d(y)}(-1)^{c(y)}
\end{equation} 
if the partition $y$ of $\{1,\dots,n_1+n_2\}$ contains $d(y)$ blocks that are hindrances. We should, indeed, count the solutions of the congruences \eqref{equaLinBlocks} for $i=0,1$ (in nonzero residues modulo $p$) and evaluate the difference. Proposition \ref{E} states that if we pay no attention to the indices $j$ that correspond to hindrances (i.e., such that $f_{B_j}=0$), the difference between numbers of solutions of $\sum_{j} f_{B_j}x_{B_j}=0$ and $\sum_{j} f_{B_j}x_{B_j}=1$ is $(-1)^{c-d(y)}$. Moreover, the values of $x_{B_j}$ where $B_j$ are hindrances can be chosen arbitrarily (from $p-1$ options each). The product of these contributions leads to \eqref{valSh}.

The formula \eqref{valSh} can be rewritten as
$$s^{i_1,i_2}(y,n_1,n_2,p)=\sum_{l=0}^{d(y)}\sum_{
\begin{array}{c}
X_1,X_2,\dots,X_l \\
\text{hindrances contained in $y$}
\end{array}}
(-1)^{c(y)-l}p^l $$
where the order of $X_1,X_2,\dots,X_l$ is irrelevant in the sum. Then we get:
\begin{align}
&C^{i_1,i_2}_0(n_1,n_2,p)-C^{i_1,i_2}_1(n_1,n_2,p)=\sum_{y\in\Pi_n} \mu(\hat 0,y)s^{i_1,i_2}(y,n_1,n_2,p)\label{CMu}\\
&=\sum_{
\begin{array}{c}
X_1,X_2,\dots,X_l \\
\text{disjoint hindrances}
\end{array}}
\sum_{
\begin{array}{c}
y\in\Pi_n \\
\text{$y$ contains $X_1,\dots,X_l$}\\
\text{as blocks}
\end{array}}
(-1)^{c(y)-l}\mu(\hat 0,y)p^l\\
&=\sum_{X_1,\dots,X_l} (-1)^{|X_1|+|X_2|+\dots+|X_l|-l}(|X_1|-1)!(|X_2|-1)!\dots(|X_l|-1)!p^l\label{Celementary}\\
&\times\smashoperator[r]{\sum_{
\begin{array}{c}
y\in\Pi_n \\  
\text{$y$ contains $X_1,\dots,X_l$}
\end{array}}}
\hspace{1em}\mu(\hat 0,y{-}X_1{-}X_2{-}\dots{-}X_l)(-1)^{c(y-X_1-X_2-\dots-X_l)}\notag
\end{align}
by factoring $\mu(\hat 0,y)$ according to the formula \eqref{exprMu}. In the last sum, $(y-X_1-X_2-\dots-X_l)$ denotes the partition $y$, where the blocks $X_1,\dots,X_l$ are removed (which is a partition of $(n_1+n_2-|X_1|-\dots-|X_l|)$ elements). 
By applying \eqref{Pascal1dCombi} to the last sum of \eqref{Celementary}, we get
\begin{multline}
\label{Celem2}
C^{i_1,i_2}_0(n_1,n_2,p)-C^{i_1,i_2}_1(n_1,n_2,p)=\\
\sum_{X_1,\dots,X_l} (|X_1|{-}1)!(|X_2|{-}1)!\dots(|X_l|{-}1)!(-1)^{n_1+n_2-l}\\
p^l(n_1{+}n_2{-}|X_1|{-}\dots{-}|X_l|)!.
\end{multline}
From \eqref{Celem2}, 
\begin{multline}
C^{i_1,i_2}_0(n_1,n_2,p)-C^{i_1,i_2}_1(n_1,n_2,p)\equiv(-1)^{n_1+n_2}(n_1+n_2)!\text{ \rm mod }p,
\end{multline}
which implies \eqref{PascalBp}.

Suppose that $\{1,\dots,n_1+n_2\}$ is not a hindrance. In order to prove \eqref{PascalC}, remark that the sum \eqref{Celem2} can be split as
\begin{multline*}
\sum_{X_1,\dots,X_l} (|X_1|{-}1)!(|X_2|{-}1)!\dots(|X_l|{-}1)!(-1)^{n_1+n_2-l}\\
p^l(n_1{+}n_2{-}|X_1|{-}\dots{-}|X_l|)!=  \\
\shoveleft{-\sum_{m=1}^{n_1+n_2} \sum_{\begin{array}{c}
X_1,X_2,\dots,X_l \\
\text{disjoint hindrances}\\
\text{not containing $m$}
\end{array}} (|X_1|{-}1)!(|X_2|{-}1)!\dots(|X_l|{-}1)!}\\
(-1)^{n_1+n_2-l-1}p^l(n_1{+}n_2{-}|X_1|{-}\dots{-}|X_l|{-}1)!
\end{multline*}
then gathered into two parts according to the values of $f_m$:
\begin{multline*}
C^{i_1,i_2}_0(n_1,n_2,p)-C^{i_1,i_2}_1(n_1,n_2,p)= \\
-n_1\sum_{\begin{array}{c}
X_1,X_2,\dots,X_l \\
\text{disjoint hindrances}\\
\text{not containing $1$}
\end{array}} (|X_1|{-}1)!(|X_2|{-}1)!\dots(|X_l|{-}1)!\\
(-1)^{n_1+n_2-l-1}p^l(n_1{+}n_2{-}|X_1|{-}\dots{-}|X_l|{-}1)!\\
-n_2\sum_{\begin{array}{c}
X_1,X_2,\dots,X_l \\
\text{disjoint hindrances}\\
\text{not containing $n_1+1$}
\end{array}} (|X_1|{-}1)!(|X_2|{-}1)!\dots(|X_l|{-}1)!\\
(-1)^{n_1+n_2-l-1}p^l(n_1{+}n_2{-}|X_1|{-}\dots{-}|X_l|{-}1)!
\end{multline*}
By identifying each sum in the last formula to the right-hand side of \eqref{Celem2} with one of the arguments $n_1$ or $n_2$ decreased by $1$, we get \eqref{PascalC}.
\end{proof} 

The numbers $\triangle^{i_1,i_2}(n_1,n_2,p)$  satisfy a similar equation.
\begin{theorem}[``Uncolored" Pascal's equation]
\label{Pascal}
Let $p$ be an odd prime, and $i_1,i_2\in\Fx$ , $n_1,n_2\in\{1,\dots,p-2\}$ such that $i_1\ne i_2$ and $n_1+n_2<p$. Then,
\begin{multline}
\label{PascalAp}
\triangle^{i_1,i_2}(n_1,n_2,p)\equiv-\triangle^{i_1,i_2}(n_1-1,n_2,p)-\triangle^{i_1,i_2}(n_1,n_2-1,p)\text{ \rm mod }p
\end{multline}
and if  $p\nmid n_1i_1+n_2i_2$, the equality
\begin{multline}
\label{PascalA}
\triangle^{i_1,i_2}(n_1,n_2,p)=-\triangle^{i_1,i_2}(n_1-1,n_2,p)-\triangle^{i_1,i_2}(n_1,n_2-1,p)
\end{multline}
holds.
\end{theorem}
\begin{proof}
Division of both sides of \eqref{PascalBp} by $n_1!n_2!$ (which is not multiple of $p$) gives \eqref{PascalAp} and division by the same number of \eqref{PascalC} gives \eqref{PascalA}.
\end{proof}
Theorem \ref{Pascal0} follows  directly from Proposition \ref{prDeltaSym} and  Theorem   \ref{Pascal}.

\section{Some properties of finite Pascal's triangles.}
\label{secArray}

\subsection{Algorithm.}
Let us define formally $\triangle^{i_1,i_2}(n_1,n_2)=0$ when one of $n_1$, $n_2$ is negative or $n_1+n_2\geqslant p$.  
Then \eqref{PascalA} is  valid for any $n_1,n_2\in\N^2$ such that $p\nmid n_1i_1+n_2i_2$. Indeed: if $n_1=0$ or $n_2=0$, the identification $A^{i_1,i_2}(n_1,n_2,p)=A(\max(n_1,n_2),p)$ implies \eqref{PascalA} via Lemma \ref{delta1d}. When $n_1+n_2=p-1$, one can use the hypothesis $X_1\cup X_2=\Fx$ and the identity $\sum\Fx=0$ to prove
$$i_1\sum X_1+i_2\sum X_2=(i_1-i_2)\sum X_1, $$
which implies $A^{i_1,i_2}_i(n_1,n_2,p)=A^{i_1-i_2,i_2}_i(n_1,0,p)$, therefore $\triangle^{i_1,i_2}(n_1,n_2)=(-1)^{n_1}$. The equation \eqref{PascalA} is valid, therefore, when $n_1+n_2=p$.

We can now prove that the functional relation \eqref{PascalA}, together with these border values,  characterizes the function $\triangle^{i_1,i_2}(\cdot,\cdot,p)$ as a function defined on $\Z_{\geqslant-1}^2$, with values in $\Z$.

\begin{theorem}
\label{AlgoTriangle}
Let $p$ be an odd prime, let $i_1,i_2$ be two distinct elements of $\{1,\dots,p-1\}$, and let $d:\Z_{\geqslant-1}^2\to\Z$ be a function such that 
\begin{align}
&d(0,0)=1, \label{trig1} \\
&d(n_1,n_2)=0 \text{ if } n_1{=}-1\text{ or }n_2{=}-1\text{ or }n_1+n_2\geqslant p,\label{trig2}\\
&d(n_1,n_2)+d(n_1-1,n_2)+d(n_1,n_2-1)=0 \text{ if } p\nmid n_1i_1+n_2i_2.\label{trig3}
\end{align}
Then, $d(n_1,n_2)=\triangle^{i_1,i_2}(n_1,n_2,p)$.
\end{theorem}
\begin{proof}
Define $\delta(n_1,n_2)=d(n_1,n_2)-\triangle^{i_1,i_2}(n_1,n_2,p)$. Then the function $\delta$ satisfies \eqref{trig2}, \eqref{trig3} and   $\delta(0,0)=0$. In order to prove the theorem we should prove that $\delta=0$.

By applying  \eqref{trig3} successively to $n_2=0$ and $n_1=1,\dots,p-1$ one proves that $\delta(0,0)=-\delta(1,0)=\delta(2,0)=\dots=\delta(p-1,0)$. By applying it to $n_1=0$ and $n_2=1,\dots,p-1$ one proves that $\delta(0,0)=-\delta(0,1)=\dots=\delta(0,p-1)$. 

Let us prove the identity $\delta(n_1,n_2)=0$ by induction on $\tilde n:=p-n_1-n_2\in\{0,\dots,p-2\}$. If $\tilde n=0$, then $\delta(n_1,n_2)=0$ as a part of the hypothesis \eqref{trig2}.

Suppose that the Theorem is proved for $\tilde n\in\{0,\dots,p-3\}$, let us prove it for  $\tilde n+1$. Denote $(n_1^S,n_2^S)$ the solution of
\begin{equation*}
\left\{\begin{array}{l}i_1n_1^S+i_2n_2^S\equiv0\text{ mod }p\\n_1^S+n_2^S=p-\tilde n\\(n_1^S,n_2^S)\in\{1,\dots,p\}^2.\end{array}\right. 
\end{equation*} 
If one applies the functional relation \eqref{trig3} to a point where $n_2=p-\tilde n-n_1$ (with the restriction $n_1\ne n_1^S$), and uses the induction hypothesis, one gets
\begin{equation}
\label{P1}
\delta(n_1-1,p-\tilde n-n_1)+\delta(n_1,p-\tilde n-n_1-1)=0.
\end{equation}
By applying \eqref{P1} successively to $n_1=1,\dots,n_1^S-1$, we prove $\delta(n_1,p-\tilde n-n_1-1)=0$ for $n_1$ in the same range $1,\dots,n_1^S-1$. If $n_1^S\geqslant p-\tilde n-1$, this concludes the step of induction. Otherwise, by applying \eqref{P1} successively to $n_1=p-\tilde n-1,\dots,n_1^S+1$ (in the decreasing order of values of $n_1$), we prove $\delta(n_1,p-\tilde n-n_1-1)=0$ for $n_1$ in the range $p-\tilde n,\dots,n_1^S$. 

This concludes the induction and proves $\delta(n_1,n_2)=0$ for all $(n_1,n_2)$.
\end{proof}

The previous proof corresponds to the  Algorithm \ref{algoAttaqueFond}, which computes the values of the function $\triangle^{a,b}(x,y,p)$ line by line.
It executes one addition per number to compute, therefore its execution time  is proportional to the size of the answer.

Given an odd prime $p$ and two distinct elements $i_1,i_2$ of $\Fx$, we are going to call the array of all values of $\triangle^{i_1,i_2}(n_1,n_2,p)$ for $n_1,n_2\geqslant 0, n_1+n_2<p$ a \textit{finite Pascal's triangle}, and we will use geometrical terminology when it seems to make exposition simpler. 

We  are going to call \it sources \rm the points $(n_1,n_2)$ such that $p|i_1n_1+i_2n_2$. Define
\begin{equation}
\label{DefForce}
f^{i_1,i_2}(n_1,n_2,p)=\triangle^{i_1,i_2}(n_1,n_2,p)+\triangle^{i_1,i_2}(n_1-1,n_2,p)+\triangle^{i_1,i_2}(n_1,n_2-1,p).
\end{equation}
The value of $f^{i_1,i_2}(n_1,n_2,p)$ (which we will call \it force\rm) is nonzero only at sources, where it can be computed using \eqref{Celem2} combined with the end of the proof of Theorem \ref{eqPC}:
\begin{multline}
\label{SourcesElem}
n_1!n_2!f^{i_1,i_2}(n_1,n_2,p)=\sum_{
\begin{array}{c}
X_1,X_2,\dots,X_l\\
\text{partition of }\{1,\dots,n_1+n_2\},\\
\forall j\ p\,\mid\,\sum_{m\in X_j}f_m
\end{array}}\\
(|X_1|-1)!(|X_2|-1)!\dots(|X_l|-1)!(-1)^{n_1+n_2-l}
p^l(n_1+n_2-|X_1|-\dots-|X_l|)!.
\end{multline}
This formula uses the notation \eqref{NotationCoeffs} in order to describe the fact that summation goes through all partitions of $\{1,\dots,n_1+n_2\}$ into hindrances.
 
The definition \eqref{DefForce} implies, by linearity of the Pascal's equation:
\begin{equation}
\label{DeltasSumOfSources}
\triangle^{i_1,i_2}(n_1,n_2,p)=\sum_{\begin{array}{c}0{\leqslant} k{\leqslant} n_1\\ 0{\leqslant} l{\leqslant} n_2 \\ p\,\mid\,i_1k{+}i_2l\end{array}} f^{i_1,i_2}(k,l,p)(-1)^{n_1+n_2-k-l}\genfrac{(}{)}{0pt}{}{n_1{+}n_2{-}k{-}l}{n_1-k}.
\end{equation}

\begin{algorithm}
\caption{Calculate a finite Pascal's triangle. Arguments $p,a,b$: $p$ prime, $0<a<b<p$}
\label{algoAttaqueFond}
\begin{algorithmic}
\State Allocate \textit{the integer array} data[\(0..p-1\)][\(0..p-1\)] (values of $\triangle^{a,b}(x,y,p)$),
\State  \textit{the boolean array} reg[\(0..p{-}1\)][\(0..p{-}1\)] (information about sources)
\For{$x=0,..,p-1,y=0,..,p-1$} 
	\State{reg\([x][y]= (a\cdot x+b\cdot y\not\equiv0\mbox{ mod }p)\)}
\EndFor
\State data\([0][0]=\)data\([p-1][0]=\)data\([0][p-1]=1\)
\State \quad\underline{\textit{resolution at the edges}}
\LINEFOR{$x=1,\dots,p{-}2$}{data\([x][p{-}1{-}x]=-\)data\([x{-}1][p{-}x]\)}
\LINEFOR{$x=1,\dots,p{-}2$}{data\([x][0]=-\)data\([x{-}1][0]\)}
\LINEFOR{$y=1,\dots,p{-}2$}{data\([0][y]=-\)data\([0][y{-}1]\)}
\State \quad\underline{\textit{resolution inside}}
\For{$n=p-2,..,1$}
	\For{$x=1,..,n-1$}
		\State \(y\gets n-x \)
		\If{\(\textrm{reg}[x][y+1]\)} 
			\State $\textrm{data}[x][y]=-\textrm{data}[x-1][y+1]-\textrm{data}[x][y+1]$
		\Else
			\State Stop the inner loop
		\EndIf
	\EndFor
	\For{$y=1,..,n-1$}
		\State \(x\gets n-y \)
		\If{\(\textrm{reg}[x+1][y]\)} 
			\State $\textrm{data}[x][y]=-\textrm{data}[x+1][y-1]-\textrm{data}[x+1][y]$
		\Else
			\State Stop the inner loop
		\EndIf
	\EndFor
\EndFor
\State \quad\underline{\textit{Print the result}}
\For{$n=0,..,p-1$}
	\For{$y=0,..,n$}
		\State Print data$[n-y][y]$, reg$[n-y][y]$
	\EndFor
	\State Print newline
\EndFor	
\end{algorithmic}
\end{algorithm}

\subsection{The case $i_1=1,i_2=2$.}
\label{SsecTriangle12}
We can find a closed formula for the numbers $\triangle^{1,2}(n_1,n_2,p)$ using the identity 
\begin{equation}\label{ConversionP12} \triangle^{1,2}(n_1,n_2,p)=\triangle^{1,2}(n_1,p-1-n_1-n_2,p).
\end{equation}
 It follows indeed from the fact that for each disjoint couple $X_1,X_2\subset\Fx$, as in the definition \eqref{defA},
\[\sum X_1+2\sum X_2=-\left(\sum X_1+2\sum(\Fx\setminus X_1\setminus X_2)\right).\]
Formula \eqref{deltaSimples} applies to at least one side of \eqref{ConversionP12} for each $(n_1,n_2)$ (and to both sides of \eqref{ConversionP12} if $n_1+2n_2=p-1$), leading to
\begin{equation}
\label{TwoHalvesTriangle}
\triangle^{1,2}(n_1,n_2,p)=\left\{\begin{array}{cl}(-1)^{n_1+n_2}\genfrac{(}{)}{0pt}{}{n_1+n_2}{n_1}&\text{ if }n_1+2n_2\leqslant p-1\\
(-1)^{n_2}\genfrac{(}{)}{0pt}{}{p-1-n_2}{n_1}&\text{ if }n_1+2n_2\geqslant p-1.\end{array}\right.
\end{equation}
Therefore, this Pascal's triangle is symmetric with respect to the axis $n_1+2n_2=p-1$.

\begin{figure}
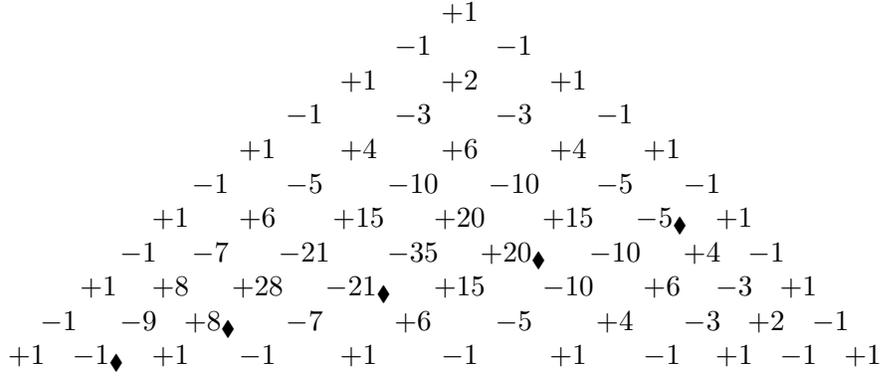

$$\begin{array}{@{\!}c@{\!}c@{\!}c@{\!}c@{\!}c@{\!}c@{\!}c@{\!}c@{\!}c@{\!}c@{\!}c@{\!}c@{\!}c@{\!}c@{\!}c@{\!}c@{\!}c@{\!}c@{\!}c@{\!}c@{\!}c@{\!}}&&&&&&&&&&+1\\&&&&&&&&&-1&&-1\\&&&&&&&&+1&&+2&&+1\\&&&&&&&-1&&-3&&-3&&-1\\&&&&&&+1&&+4&&+6&&+4&&+1\\
&&&&&-1&&-5&&-10&&-10&&-5&&-1\\
&&&&+1&&+6&&+15&&+20&&+15&&-5_\Kv&&+1\\
&&&-1&&-7&&-21&&-35&&+20_\Kv&&-10&&+4&&-1\\
&&+1&&+8&&+28&&-21_\Kv&&+15&&-10&&+6&&-3&&+1\\
&-1&&-9&&+8_\Kv&&-7&&+6&&-5&&+4&&-3&&+2&&-1\\
+1&&-1_\Kv&&+1&&-1&&+1&&-1&&+1&&-1&&+1&&-1&&+1
\end{array}
$$
\caption{Coefficients of $\prod_{j=1}^{10}\left(X+\zeta_{11}^{j}Y+\zeta_{11}^{2j}Z\right)$}
\label{triangle12}
\end{figure}

One can deduce \eqref{NormLucas} from \eqref{TwoHalvesTriangle} in the following way: by \eqref{normPolyn2},
\begin{multline}
\label{NormSources12}
\prod_{j=1}^{j=p-1}\left(1+\zeta_p-\zeta_p^2\right)=\sum_{n_1,n_2\in\N}(-1)^{n_2}\triangle^{1,2}(n_1,n_2,p)\\
=\sum_{n_1,n_2\in\N}(-1)^{n_2-1}(\triangle^{1,2}(n_1{-}1,n_2,p)+\triangle^{1,2}(n_1,n_2{-}1,p)-f^{1,2}(n_1,n_2,p))\\
=\sum_{n_1,n_2\in\N}((-1)^{n_2-1}\triangle^{1,2}(n_1,n_2-1,p) -(-1)^{n_2}\triangle^{1,2}(n_1-1,n_2,p)\\ + (-1)^{n_2}f^{1,2}(n_1,n_2,p))\\
=\sum_{n_1,n_2\in\N}(-1)^{n_2}f^{1,2}(n_1,n_2,p)
\end{multline}
because massive cancellation occurs in the sum of differences of values of the function $(-1)^y\triangle^{1,2}(x,y,p)$.

Suppose $n_1,n_2>0$ and $n_1+2n_2=p$ (therefore $n_1$ is odd). Then
\begin{multline}
\label{source12}
f^{1,2}(n_1,n_2,p)=\triangle^{1,2}(n_1-1,n_2,p)+\triangle^{1,2}(n_1,n_2-1,p)+\triangle^{1,2}(n_1,n_2,p)\\
=(-1)^{n_2}\genfrac{(}{)}{0pt}{}{n_1+n_2-1}{n_1-1}+2(-1)^{n_2}\genfrac{(}{)}{0pt}{}{n_1+n_2-1}{n_1}\\
=(-1)^{n_2}\left(\genfrac{(}{)}{0pt}{}{n_1+n_2}{n_1}+\genfrac{(}{)}{0pt}{}{n_1+n_2-1}{n_1} \right)\\
=(-1)^{n_2}\left(\genfrac{(}{)}{0pt}{}{p-n_2}{n_2}+\genfrac{(}{)}{0pt}{}{p-n_2-1}{n_2-1} \right).
\end{multline}

The absolute value of \eqref{source12} can be interpreted as the number of ways to put $n_2$ identical disjoint dominoes  on a discrete circle of length $p$. Indeed (see also \cite{Lucas}), for any $k\leqslant \frac{p-1}{2}$
\begin{multline}
\#\{k\text{ disjoint dominoes on a circle of length }p\}\\
=\#\{k\text{ disjoint dominoes on a line segment of length }p\}\\+\#\{k-1\text{ disjoint dominoes on a line segment of length }p-2\}\\
=\genfrac{(}{)}{0pt}{}{p-k}{k}+\genfrac{(}{)}{0pt}{}{p-k-1}{k-1}.
\end{multline}
The sum \eqref{NormSources12} contains three terms not covered by the hypotheses of \eqref{source12}: these correspond to $n_1{=}n_2{=}0$, $n_1{=}p,n_2{=}0$, $n_1{=}0,n_2{=}p$ and they equal respectively $1$, $1$ and $-1$. The overall contribution of these terms  can be identified to the number of ways to put $0$ dominoes on a discrete circle of length $p$. Therefore, the norm \eqref{NormSources12} equals to the number of ways to put any number of identical disjoint dominoes  on a discrete circle of length $p$, which is proved in \cite{Lucas} to be $L_p$.

For example, if $p=11$, the numbers are those of Figure \ref{triangle12} ($\Kv$ denotes a source).

\subsection{Application: an identity for binomial coefficients.}
The formulas \eqref{normPolyn2} and \eqref{NormSources12} have another application. As $1-\zeta_p+\zeta_p^2=\frac{1+\zeta_p^3}{1+\zeta_p}$, we get in a similar way to \eqref{NormSources12}:
\begin{multline}
\label{NormPMP}
1=\prod_{j=1}^{j=p-1}\left(1-\zeta_p+\zeta_p^2\right)=\sum_{n_1,n_2\in\N}(-1)^{n_1}\triangle^{1,2}(n_1,n_2,p)\\=\sum_{n_1,n_2\in\N}(-1)^{n_1}f^{1,2}(n_1,n_2,p).
\end{multline}
We further get:
\begin{multline}
\label{NormPMP2}
1=1+\sum_{n_1,n_2\in\N^*}(-1)^{n_1}f^{1,2}(n_1,n_2,p)=1-\sum_{n_1,n_2\in\N^*}f^{1,2}(n_1,n_2,p).
\end{multline}
The formula \eqref{source12} leads to the following combinatorial identity\footnote{The previous proof implies \eqref{formuleBinom} in the case of prime $p{\geqslant}5$. The Zeilberger's algorithm (implemented in Maple 17, see also Chapter $6$ of the book \cite{AegalB}) generalizes it for any $p{\geqslant}5$ congruent to $1$ or $5$ modulo $6$}:
\begin{equation}
\label{formuleBinom}
\sum_{k=1}^{\frac{p-1}{2}}(-1)^k\left(\genfrac{(}{)}{0pt}{}{p-k}{k}+\genfrac{(}{)}{0pt}{}{p-k-1}{k-1}\right)=0.
\end{equation}

\subsection{Second application: expression for a symmetric polynomial.}
We can formulate an expression for an arbitrary symmetric polynomial of the numbers $(1+\zeta_p^j-\zeta_p^{2j})$ which is:
\begin{theorem}
\label{PSymetriquePPM}
Let $p\geqslant5$ be prime and $\delta\in\{0,\dots,p-2\}$ an integer. Then $\sigma_{p-1-\delta,(j=1,\dots,p-1)}(1+\zeta_p^j-\zeta_p^{2j})$ (see the notation of Proposition \ref{prDeltaSym}) equals $\genfrac{(}{)}{0pt}{}{p-1}{\delta}$ plus the sum of ``weights" of ways of putting a number $n{>}0$ of disjoint dominoes on a discrete circle of length $p$, the weights being $\genfrac{(}{)}{0pt}{}{n-1}{\delta}$.

As a consequence, $\sigma_{p-1-\delta}(1+\zeta-\zeta^2)>0$ and $\sigma_{p-1-\delta}(1+\zeta-\zeta^2)\equiv\genfrac{(}{)}{0pt}{}{p-1}{\delta}\text{ \rm mod }p.$
\end{theorem}
\begin{proof}
By Proposition \ref{prDeltaSym}, we get a similar expression to \eqref{NormSources12}
\begin{multline}
\label{symPPM0}
\sigma_{p-1-\delta,(j=1,\dots,p-1)}\left(1+\zeta_p^{j}+\zeta_p^{2j}\right)
\\=\sum_{n_1,n_2\in\N}(-1)^{n_2}\genfrac{(}{)}{0pt}{}{p-1-n_1-n_2}{\delta}\triangle^{1,2}(n_1,n_2,p)
\\=\sum_{\tilde n=1}^p\sum_{\begin{array}{c}n_1,n_2\\n_1+n_2=p-\tilde n\end{array}}(-1)^{n_2}\genfrac{(}{)}{0pt}{}{\tilde n-1}{\delta}\\(-\triangle^{1,2}(n_1-1,n_2,p)-\triangle^{1,2}(n_1,n_2-1,p)+f^{1,2}(n_1,n_2,p))
\\=\sum_{\tilde n=1}^p\sum_{\begin{array}{c}n_1,n_2\\n_1+n_2=p-\tilde n\end{array}}(-1)^{n_2}\genfrac{(}{)}{0pt}{}{\tilde n-1}{\delta}f^{1,2}(n_1,n_2,p).
\end{multline}
The identity \eqref{source12} leads to
\begin{multline}
\label{symPPMDominos}
\sigma_{p-1-\delta,(j=1,\dots,p-1)}\left(1+\zeta_p^{j}+\zeta_p^{2j}\right)
\\=\genfrac{(}{)}{0pt}{}{p-1}{\delta}+\sum_{n_2=1}^{\frac{p-1}{2}}\genfrac{(}{)}{0pt}{}{n_2-1}{\delta}\left(\genfrac{(}{)}{0pt}{}{p-n_2}{n_2}+\genfrac{(}{)}{0pt}{}{p-n_2-1}{n_2-1}\right)
\end{multline}
and the discussion that follows the formula \eqref{source12} identifies each number $(-1)^{n_2}f^{1,2}(n_1,n_2,p)$ as the number of ways to put $n_2$ disjoint dominoes on a discrete circle of length $p$.  
\end{proof}
\vfill\eject

\subsection{The case $i_1=1,i_2=3$.}
\label{subsecPascal13}
\begin{figure}
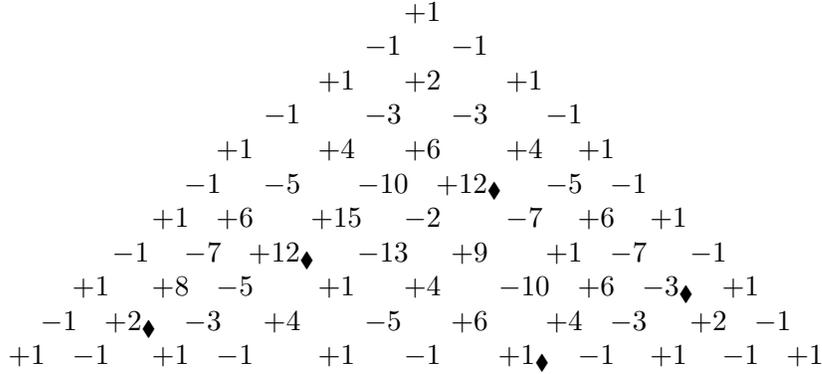

$$\begin{array}{@{\!}c@{\!}c@{\!}c@{\!}c@{\!}c@{\!}c@{\!}c@{\!}c@{\!}c@{\!}c@{\!}c@{\!}c@{\!}c@{\!}c@{\!}c@{\!}c@{\!}c@{\!}c@{\!}c@{\!}c@{\!}c@{\!}}
&&&&&&&&&&+1\\&&&&&&&&&-1&&-1\\&&&&&&&&+1&&+2&&+1\\&&&&&&&-1&&-3&&-3&&-1\\&&&&&&+1&&+4&&+6&&+4&&+1\\
&&&&&-1&&-5&&-10&&+12_\Kv&&-5&&-1\\
&&&&+1&&+6&&+15&&-2&&-7&&+6&&+1\\
&&&-1&&-7&&+12_\Kv&&-13&&+9&&+1&&-7&&-1\\
&&+1&&+8&&-5&&+1&&+4&&-10&&+6&&-3_\Kv&&+1\\
&-1&&+2_\Kv&&-3&&+4&&-5&&+6&&+4&&-3&&+2&&-1\\
+1&&-1&&+1&&-1&&+1&&-1&&+1_\Kv&&-1&&+1&&-1&&+1\\
\end{array}
$$
\caption{Coefficients of $\prod_{j=1}^{10}\left(X+\zeta_{11}^{j}Y+\zeta_{11}^{3j}Z\right)$}
\label{triangle13}
\end{figure}
In this case the formula
\begin{equation}
\label{Equivalence13}
\triangle^{1,3}(n_1,n_2,p)=\triangle^{2,3}(n_1,p-1-n_1-n_2,p)  
\end{equation}
is analogous to \eqref{ConversionP12} and implies
\begin{equation}
\label{BinomialZones13}
\triangle^{1,3}(n_1,n_2,p)=\left\{\begin{array}{@{}c@{}l@{}l}(-1)^{n_1+n_2}\genfrac{(}{)}{0pt}{}{n_1+n_2}{n_1}&\text{ if }&n_1+3n_2\leqslant p-1\\
(-1)^{n_2}\genfrac{(}{)}{0pt}{}{p-1-n_2}{n_1}&\text{ if }&n_1+3n_2\geqslant 2p-2\text{ or }\\&&n_1+3n_2=2p-4,\end{array}\right.
\end{equation} 
therefore, in two regions, the coefficients of the triangle are identical to the previous case. 

The coefficients in the middle region can be calculated using the general formula \eqref{DeltasSumOfSources}. Let us specify different quantities used there, namely the position of sources and the associated forces. The sources are the integer points situated on two lines: the \it upper \rm line with equation $n_1+3n_2=p$ and the \it lower \rm line with equation $n_1+3n_2=2p$. One can see that the number of integer points on the upper line of sources is
\begin{equation}
\label{UpperLine}
\#\left\{\begin{array}{c}0<n_1<p\\0<n_2<p\\n_1+3n_2=p \end{array}\right\} =\left\lfloor\frac{p}{3}\right\rfloor
\end{equation}
and the number of integer points on the lower line is
\begin{equation}
\label{LowerLine}
\#\left\{\begin{array}{c}0<n_1<p\\0<n_2<p\\n_1+3n_2=2p \end{array}\right\} =\textrm{rnd}(\frac{p}{6}),
\end{equation}
the closest integer to $\frac{p}{6}$.

If $(n_1,n_2)$ is a point on the upper line of sources, the value of \\$f^{1,3}(n_1,n_2,p)$ has a simple expression given by \eqref{SourcesElem}:
\begin{equation}
\label{SourceHaut13}
f^{1,3}(n_1,n_2,p)=\frac{(n_1+n_2-1)!p}{n_1!n_2!}
\end{equation}
because the sum consists of the single term associated to \\${X{=}\{1,\dots,n_1{+}n_2\}}$. Under the same hypotheses, \eqref{DeltasSumOfSources} implies
\begin{equation}
\label{DroiteHaut13}
\triangle^{1,3}(n_1,n_2,p)=\frac{(n_1+n_2-1)!p}{n_1!n_2!}-\genfrac{(}{)}{0pt}{}{n_1+n_2}{n_1}=2\genfrac{(}{)}{0pt}{}{n_1{+}n_2{-}1}{n_1}.
\end{equation}

In any point $(n_1,n_2)$ such that $p\leqslant n_1+3n_2<2p$, the formula \eqref{DeltasSumOfSources} takes the following form:
\begin{multline}
\label{DeltasMilieu}
\triangle^{1,3}(n_1,n_2,p)=(-1)^{n_1+n_2}\genfrac{(}{)}{0pt}{}{n_1+n_2}{n_1}\\+\sum_{\begin{array}{c}0{<} k{\leqslant} n_1\\ 0{<}l{\leqslant} n_2 \\ p=i_1k{+}i_2l\end{array}} f^{1,3}(k,l,p)(-1)^{n_1+n_2-k-l}\genfrac{(}{)}{0pt}{}{n_1{+}n_2{-}k{-}l}{n_1-k}.
\end{multline}

We can also compute a simple expression for the forces of sources on the lower line. Suppose that $n_1,n_2>0$ and $n_1+3n_2=2p$. Then, by \eqref{DefForce} and \eqref{Equivalence13},
\begin{multline}
\label{ConversionForcesBas}
f^{1,3}(n_1,n_2,p)=\triangle^{1,3}(n_1,n_2,p)+\triangle^{1,3}(n_1-1,n_2,p)+\triangle^{1,3}(n_1,n_2-1,p)\\
=f^{2,3}(n_1,p-n_1-n_2,p).
\end{multline}
By \eqref{SourcesElem} (the sum, once again, consists of a single term because \\$2n_1+3(p{-}n_1{-}n_2){=}p$),
\begin{equation}
\label{ForcesBas}
f^{2,3}(n_1,p-n_1-n_2,p)=\frac{(-1)^{n_2}(p-n_2-1)!p}{n_1!(p-n_1-n_2)!}.
\end{equation}  

For example, if $p=11$, the numbers are those of Figure \ref{triangle13} ($\Kv$ denotes a source).

\end{document}